\documentclass[12pt]{article}
\usepackage[margin=1in]{geometry}
\usepackage{amsthm, amsmath,amsfonts,amssymb,euscript,hyperref,graphics,color,slashed,mathrsfs}
\usepackage{graphicx}
\usepackage{comment}
\usepackage{import}
\usepackage{tikz}
\usepackage{latexsym}
\usepackage{mathtools}
\usepackage{appendix}
\usepackage{stmaryrd}

\makeatletter
\newsavebox{\@brx}
\newcommand{\llangle}[1][]{\savebox{\@brx}{\(\m@th{#1\langle}\)}%
	\mathopen{\copy\@brx\kern-0.5\wd\@brx\usebox{\@brx}}}
\newcommand{\rrangle}[1][]{\savebox{\@brx}{\(\m@th{#1\rangle}\)}%
	\mathclose{\copy\@brx\kern-0.5\wd\@brx\usebox{\@brx}}}
\makeatother

\makeatletter
\newcommand{\subjclass}[2][1991]{%
	\let\@oldtitle\@title%
	\gdef\@title{\@oldtitle\footnotetext{#1 \emph{Mathematics subject classification.} #2}}%
}
\newcommand{\keywords}[1]{%
	\let\@@oldtitle\@title%
	\gdef\@title{\@@oldtitle\footnotetext{\emph{Key words and phrases.} #1}}%
}
\makeatother

\newtheorem{theorem}{Theorem}[section]
\newtheorem*{theorem*}{Theorem}
\newtheorem{lemma}[theorem]{Lemma}

\newtheorem{corollary}[theorem]{Corollary}
\theoremstyle{remark}
\newtheorem{definition}[theorem]{Definition}
\newtheorem{remark}[theorem]{Remark}

\newtheorem*{acknowledgements}{\bf Acknowledgements}

\numberwithin{equation}{section}

\begin{document}
\title {On the irrationality of certain $p$-adic zeta values}
	
\author{Li Lai, Cezar Lupu, Johannes Sprang}
\date{}
	
\maketitle
	
\begin{abstract}
A famous theorem of Zudilin states that at least one of the Riemann zeta values $\zeta(5), \zeta(7), \zeta(9), \zeta(11)$ is irrational. 
In this paper, we establish the $p$-adic analogue of Zudilin's theorem. 
As a weaker form of our result, it is proved that for any prime number $p \geqslant 5$ there exists an odd integer $i$ in the interval $[3,p+p/\log p+5]$ such that the $p$-adic zeta value $\zeta_p(i)$ is irrational.
\end{abstract}

\section{Introduction}

The Riemann zeta function is defined by the absolutely convergent series
\[
\displaystyle\zeta(s)=\sum_{n=1}^{\infty}\frac{1}{n^s},\quad\operatorname{Re}s>1.
\]
The values of $\zeta(s)$ at positive even integers are non-zero rational multiples of powers of $\pi$, \,---\,a result which goes back to Euler in 1735.
On the other hand, the values of $\zeta(s)$ when $s\geqslant 3$ is odd (referred to as \emph{odd zeta values}) are still far away from being fully understood. 
At this point, the state of the art is as follows. 
A groundbreaking result of Ap{\'e}ry \cite{Ape1979} from 1979 states that $\zeta(3)$ is irrational. 
The next breakthrough was achieved in 2001 by Ball and Rivoal \cite{BR2001, Riv2000} who proved that the dimension of the $\mathbb{Q}$-linear span of $\zeta(3),\zeta(5), \ldots, \zeta(s)$ is at least $((1-\varepsilon)/\log 2) \cdot \log s$ for any $\varepsilon>0$ and any sufficiently large odd integer $s \geqslant s_0(\varepsilon)$. 
In particular, there exist infinitely many irrationals among odd zeta values. 
Also, another amazing result is due to Zudilin \cite{Zud2001} who proved the following.
\begin{theorem*}[Zudilin, 2001]
At least one of $\zeta(5),\zeta(7),\zeta(9),\zeta(11)$ is irrational.
\end{theorem*}
So far, the irrationality of $\zeta(5)$ seems beyond reach, but recently \cite{BZ2022+} there has been an attempt in this regards by Brown and Zudilin who constructed infinitely many effective rational approximations $m/n$ for $\zeta(5)$ satisfying
$$\displaystyle 0<\left|\zeta(5)-\frac{m}{n}\right|<\frac{1}{n^{0.86}}.$$
 These effective pairs $(m, n)$ come from the solution to a linear difference equation with polynomial coefficients. Their method involves a certain family of cellular integrals which are period integrals in the moduli space $\mathcal{M}_{0, 8}$ of curves of genus zero with eight marked points.
For recent asymptotic partial results on the irrationality or linear independence of odd zeta values, see \cite{ Fis2021+, FSZ2019, Lai2025-II+, LY2020}.

Let us turn our attention to the $p$-adic world. 
The $p$-adic zeta value $\zeta_p(s)$, for an integer $s \geqslant 2$, is characterized by 
\[
\zeta_p(s) = \lim_{\substack{k\to s \;\text{$p$-adically}\\k \in \mathbb{Z}_{<0}, \; k\equiv s \pmod{p-1}}} \zeta(k) \in \mathbb{Q}_p.
\]
In other words, the $p$-adic zeta value $\zeta_p(s)$ is a $p$-adic limit of special values of the Riemann zeta function at negative integers. 
It is known that $\zeta_p(s)=0$ for any positive even integer $s$, but much less is known for $\zeta_p(s)$ when $s \geqslant 3$ is odd.

In 2005, Calegari \cite{Cal2005} established a $p$-adic analogue of Ap{\'e}ry's theorem by proving that $\zeta_p(3)$ is irrational for $p=2,3$. Calegari made use of $p$-adic modular forms, which is reminiscent of Beukers' proof of Ap\'ery's theorem \cite{Beu1987}. 
See also Beukers \cite{Beu2008} for an alternative proof of the irrationality of $\zeta_2(3)$ and $\zeta_3(3)$ using Pad\'e approximation.

In 2020, the third author obtained the $p$-adic analogue of the Ball--Rivoal theorem. 
More precisely, the third author \cite{Spr2020} proved that the dimension of the $\mathbb{Q}$-linear span of $\zeta_p(3),\zeta_p(5),\ldots,\zeta_p(s)$ is at least $((1-\varepsilon)/(2\log2))\cdot\log s$ for any prime $p$, any $\varepsilon>0$, and any sufficiently large odd integer $s \geqslant s_0(p,\varepsilon)$. 
In \cite{Spr2020}, the third author made use of Volkenborn integrals to construct linear forms in $1$ and $p$-adic zeta values.
This method turns out to be a systematic tool for irrationality proofs concerning $p$-adic zeta values.
It inspired the first author \cite{Lai2025} to prove that at least one of $\zeta_2(7),\zeta_2(9),\zeta_2(11),\zeta_2(13)$ is irrational.
In addition, the first and third authors \cite{LS2023+} used Volkenborn integrals to prove certain asymptotic results on the number of irrationals among $p$-adic zeta values for any prime $p$.

Recently, Calegari, Dimitrov and Tang \cite{CDT2020+} proved the irrationality of $\zeta_2(5)$ using their arithmetic holonomicity criterion. 
This result has been independently proved by the first and third authors, along with Zudilin, in \cite{LSZ2025+} using an Ap\'ery-like approach. 
Note that for any prime $p \geqslant 5$, the irrationality of $\zeta_p(3)$ still remains an open question.

In this paper, we prove a $p$-adic analogue of Zudilin's theorem \cite{Zud2001} for primes $p\geqslant 5$. 
Let $\psi(x)$ be the digamma function and $\gamma$ be the Euler–Mascheroni constant. 
Our main result is as follows. 
	
\begin{theorem}\label{main_thm}
For any prime number $p \geqslant 5$, there exists an odd integer $i$ in the interval $[3,c_p]$ such that the $p$-adic zeta value $\zeta_p(i)$ is irrational, where
\[ 
c_p = p + \frac{p-1-\varpi_p}{\frac{p}{p-1}\log p - 1 -\log 2} 
\]
and 
\[ 
\varpi_p = \psi\left( \frac{1}{p} \right)  + 2p-1 + \gamma + p\left( \log p - \sum_{j=1}^{p} \frac{1}{j} \right). 
\]
\end{theorem}
	
\bigskip
	
Table \ref{table} gives the numerical values of $c_p$ for some small primes $p \geqslant 5$.
\begin{table}[h]
\centering
\begin{tabular}{|c|c|c|}
\hline
$p$ & $c_p$ & the greatest odd integer $\leqslant c_p$ \\
\hline
$5$ &  $14.6698\ldots$   &  13   \\
\hline
$7$ &  $14.4779\ldots$  &  13   \\
\hline 
$11$ &  $18.0949\ldots$  &  17 \\
\hline  
$13$  &  $20.2567\ldots$  &  19 \\
\hline
$17$  &  $24.7553\ldots$ & 23 \\
\hline
$19$  &  $27.0426\ldots$ &  27\\
\hline
$\vdots$  &  $\vdots$  & $\vdots$ \\
\hline
$101$ &  $120.8043\ldots$ & 119 \\
\hline
\end{tabular}
\caption{numerical values of $c_p$}
\label{table}
\end{table}

For $p\geqslant 5$, we have the estimate (the greatest odd integer not exceeding $c_p$) $\leqslant p+p/\log p+5$ (see Remark \ref{rem:bound_cp}), which leads to the following.

\begin{corollary}
For any prime number $p \geqslant 5$, there exists an odd integer $i$ in the interval $[3,p+p/\log p+5]$ such that the $p$-adic zeta value $\zeta_p(i)$ is irrational. 
\end{corollary}

The structure of this paper is as follows. 
In \S \ref{sec:Prelim}, we introduce a basic irrationality criterion for $p$-adic numbers and recall the definition of $p$-adic zeta values.
In \S \ref{sec:Rational}, we construct a sequence of rational functions $R_n(t)$.
In \S \ref{sec:LinearForms}, we construct a sequence of linear forms $S_n$ in $1$ and $p$-adic zeta values.
In \S \ref{sec:Arithmetic}, we prove the arithmetic properties of the coefficients appearing in the linear forms.
In \S \ref{sec:p-adic_norm}, we estimate the $p$-adic norm of $S_n$.
In \S \ref{sec:Archimedean}, we prove certain Archimedean estimates.
Finally, we prove Theorem \ref{main_thm} in \S \ref{sec:proof}.

\section{Preliminaries}
\label{sec:Prelim}

\subsection{An irrationality criterion}
	
We first state an elementary irrationality criterion in the following, which is sufficient to prove Theorem \ref{main_thm}.

\begin{lemma}[{\cite[Lemma 2.1]{Lai2025}}]
\label{lem:irrationalityCrit}
Let $\xi_0,\xi_1,\ldots,\xi_s\in \mathbb{Q}_p$ and
\[
L_n(X_0,X_1,\ldots,X_s):=l_{0,n}X_0+l_{1,n}X_1+\cdots + l_{s,n}X_s\in \mathbb{Z}[X_0,X_1,\dots, X_s] 
\]
be a sequence of linear forms. 
Assume that there is an unbounded subset $I\subseteq\mathbb{N}$ such that:
\begin{enumerate}
\item[\textup{(1)}]
\label{lem:irrationalityCrit:i} $ \max_{0\leqslant i\leqslant s}|l_{i,n}|\cdot |L_n(\xi_0,\xi_1,\ldots,\xi_s)|_p\rightarrow 0$ as $n \in I$ and $n\to \infty$,
\item[\textup{(2)}]
\label{lem:irrationalityCrit:ii} $L_n(\xi_0,\xi_1,\ldots,\xi_s) \neq 0$ for all $n\in I$.
\end{enumerate}
Then, at least one of $\xi_0,\xi_1,\ldots,\xi_s$ is irrational. 
\end{lemma}
		
\subsection{Volkenborn integrals and the Bernoulli functional}
In the following, we recall basic facts about Volkenborn integrals. 
For more details, we refer the reader to \cite[Chapter 5]{Rob2000} or \cite[\S 55]{Sch2006}. 
The \emph{Volkenborn integral} of a continuous function $f\colon \mathbb{Z}_p \to \mathbb{Q}_p$ is defined as
\[
\int_{\mathbb{Z}_p}f(t) \mathrm{d}t:= \lim_{n \to \infty} \frac{1}{p^n}\sum_{0\leqslant k<p^n} f(k),
\]
if this limit exists. 
In this case, $f$ is called \emph{Volkenborn integrable}. 
For example, every strictly differentiable function is Volkenborn integrable. 
The Volkenborn integral has the following behavior under translations:
\begin{lemma}[{\cite[\S 5.3, Prop. 2]{Rob2000}}]
\label{lem:Volkenborn_translation}
Let $m$ be a positive integer and $f\colon\mathbb{Z}_p\to \mathbb{Q}_p$ a strictly differentiable function, then
\[
\int_{\mathbb{Z}_p} f(t+m)\mathrm{d}t= \int_{\mathbb{Z}_p} f(t)\mathrm{d}t+\sum_{\nu=0}^{m-1}f'(\nu).
\]
\end{lemma}

Let us recall that the $n$-th Bernoulli polynomial $\mathbb{B}_n(X)$ for a non-negative integer $n$ is given by the following generating function:
\[
\frac{te^{X\cdot t}}{e^t-1}=\sum_{n=0}^\infty \mathbb{B}_n(X)\frac{t^n}{n!}.
\]
The constant term of $\mathbb{B}_n(X)$ gives the $n$-th Bernoulli number $B_n=\mathbb{B}_n(0)$.

\begin{lemma}[{\cite[Ch. 5, \S 5.4]{Rob2000}}]
\label{lem:Volkenborn_Bernoulli}
Let $n$ be a non-negative integer and $x\in \mathbb{Q}_p$. 
Then
\[
\int_{\mathbb{Z}_p} (t+x)^n \mathrm{d}t= \mathbb{B}_n(x),
\]
In particular, for $x=0$ the above integral gives the $n$-th Bernoulli number $B_n$.
\end{lemma}
	
Let us briefly recall the definition of the $n$-th Bernoulli functional on the space of overconvergent functions, compare \cite{LS2023+}. 
The $\mathbb{Q}_p$-algebra of \emph{overconvergent functions} on $\mathbb{Z}_p$ is given by
\[
C^{\dagger}(\mathbb{Z}_p,\mathbb{Q}_p):=\varinjlim_{\rho>1} C^{\text{an}}_\rho(\mathbb{Z}_p,\mathbb{Q}_p),
\]
where $C^{\text{an}}_\rho(\mathbb{Z}_p,\mathbb{Q}_p)$ denotes the $\mathbb{Q}_p$-Banach algebra 
\[
C^{\text{an}}_\rho(\mathbb{Z}_p,\mathbb{Q}_p):=\left\{ f\in C(\mathbb{Z}_p,\mathbb{Q}_p) \mid f(x)=\sum_{k=0}^\infty a_kx^k \text{ with } |a_k|_p\rho^k\to 0 \text{ as }k\to \infty \right\}
\]
of $\mathbb{Q}_p$-analytic functions of radius of convergence at least $\rho$ with its usual maximum norm. 

For $n\geqslant 1$, the \emph{$n$-th Bernoulli functional} $\mathcal{L}_n\colon C^{\dagger}(\mathbb{Z}_p,\mathbb{Q}_p)\to \mathbb{Q}_p$ is defined by
\[
\mathcal{L}_n\colon C^{\dagger}(\mathbb{Z}_p,\mathbb{Q}_p)\to \mathbb{Q}_p,\quad f=\sum_{k=0}^\infty a_kt^k\mapsto \mathcal{L}_n(f) :=\sum_{k=0}^\infty n\cdot a_k \frac{B_{k+n}}{k+n}.
\]
Note that this functional is well-defined by the overconvergence of $f$ together with the von Staudt-Clausen congruences on Bernoulli numbers.
	
The following lemma shows that the first Bernoulli functional of the derivative of an overconvergent function can be computed by a Volkenborn integral:
\begin{lemma}[{\cite[Lemma 2.7]{LS2023+}}]
\label{lem:comparison_Volkenborn_Bernoulli}
For any $f\in C^{\dagger}(\mathbb{Z}_p,\mathbb{Q}_p)$ we have the formula
    \begin{equation}\label{eq:comparison_Volkenborn_Bernoulli}
    \mathcal{L}_1(f')=\int_{\mathbb{Z}_p} f(t) \mathrm{d}t - f(0).
    \end{equation}
\end{lemma}

\subsection{$p$-adic Hurwitz zeta functions}
Let $p$ be an odd prime.
In this subsection, we recall some basic facts about $p$-adic Hurwitz zeta functions. 
Since $p$ is assumed to be an odd prime, the units $\mathbb{Z}_p^\times$ of the $p$-adic integers decompose canonically
\[
\mathbb{Z}_p^\times \xrightarrow{\sim} \mu_{p-1}(\mathbb{Z}_p)\times (1+p\mathbb{Z}_p).
\]
Here, $\mu_n(R)$ denotes the group of $n$-th roots of unity in a ring $R$. 
The canonical projection
\[
\omega \colon \mathbb{Z}_p^\times \rightarrow \mu_{p-1}(\mathbb{Z}_p)
\]
is called the \emph{Teichm\"uller character}, which can be extended multiplicatively to a map
\[
\mathbb{Q}_p^\times\rightarrow \mathbb{Q}_p^\times,
\]
by setting
\[
\omega(x):=p^{v_p(x)}\omega(x/p^{v_p(x)}).
\]
We define $\langle x\rangle:=x/\omega(x)$ for $x\in\mathbb{Q}_p^\times$. 
For $x\in\mathbb{Q}_p$ with $|x|_p>1$, there is a unique $p$-adic meromorphic function 
\[
\zeta_p(\cdot,x)\colon \{s\in\mathbb{C}_p\setminus\{1\} \mid |s|_p<p^{1-\frac{1}{p-1}}\} \to \mathbb{C}_p,
\]
called the \emph{$p$-adic Hurwitz zeta function}, such that 
\[
\zeta_p(1-n,x)=-\omega(x)^{-n}\frac{\mathbb{B}_n(x)}{n}, \quad (n\geqslant 2).
\]
Explicitly, the $p$-adic Hurwitz zeta function can be defined in terms of Volkenborn integrals as follows:
\begin{equation}
\label{eq:zeta_Volkenborn}
\zeta_p(s,x)=\frac{1}{s-1}\int_{\mathbb{Z}_p}\langle t+x\rangle^{1-s}\mathrm{d}t,
\end{equation}
see \cite[Def. 11.2.5]{Coh2007} for more details. 
The following lemma  relates the $p$-adic Hurwitz zeta function to $p$-adic zeta values:
\begin{lemma}[{\cite[Lemma 2.10]{LS2023+}}]
\label{lem:padicHurwitz_zeta}
Let $D$ and $i$ be positive integers with $p|D$ and $i\geqslant 2$, then
\[
D^i\cdot \zeta_p(i)=\sum_{\substack{1\leqslant j\leqslant D \\ \gcd(j,p)=1 }} \omega\left(\frac{j}{D}\right)^{1-i}\zeta_p\left(i,\frac{j}{D}\right).
\]
\end{lemma}

\section{Rational functions}
\label{sec:Rational}

From now on, we fix any prime $p \geqslant 5$. Recall that the Pochhammer symbol $(\alpha)_k$ is defined by 
\[ 
(\alpha)_k := \alpha(\alpha+1)\cdots(\alpha+k-1) 
\]
for any positive integer $k$.
	
Fix a positive integer $s$. 
Let 
\begin{equation}
\label{defi_N_0_and_i_0}
N_0 = v_p(p-1+s) \quad\text{and}\quad M_0 = p^{2+N_0}s -1.
\end{equation}
Note that 
\begin{equation}
\label{upper_bound_for_i_0}
M_0 \leqslant p^2(p-1+s)s-1 < (p+s)^4.
\end{equation}

\begin{definition}
\label{def:R_n(t)}
We define for any positive integer $n > (p+s)^4$ the rational function $R_n(t) \in \mathbb{Q}(t)$ by
\[ 
R_n(t) := p^{pn} \cdot n!^s \cdot t^{M_0} \cdot \frac{\prod_{j=1}^{p-1} \left( t+\frac{j}{p} \right)_n}{(t)_{n+1}^{p-1+s}}.  
\]
\end{definition}
	
For a rational function $R(t) = P(t)/Q(t)$, where $P(t)$, $Q(t)$ are polynomials in $t$, we define the degree of $R(t)$ by $\deg R := \deg P - \deg Q$. 
By Definition \ref{def:R_n(t)} and Eq. \eqref{upper_bound_for_i_0}, we have
\[ 
\deg R_n = M_0 -(p-1) - s(n+1)  \leqslant -2  
\]
for every $n >(p+s)^4$.
	
Note that $M_0 \geqslant p^2s-1>p-1+s$, so $t=0$ is not a pole of $R_n(t)$. 
Recall that the $p$-adic logarithm $\log_p$ is defined on $\{ 1+x \in \mathbb{Q}_p ~\mid~ |x|_p < 1  \}$ by
\[ 
\log_p(1+x) := \sum_{j=1}^{\infty} (-1)^{j-1} \frac{x^{j}}{j}, \quad |x|_p < 1. 
\]
The function $f(t) = \log_p\langle t \rangle$ is defined on $\mathbb{Q}_p^{\times}$ and $f'(t) = 1/t$ for every $t \in \mathbb{Q}_p^{\times}$.
	
\begin{definition}
\label{def:tildaR_n(t)}
For any integer $n > (p+s)^4$, we denote the partial fraction decomposition of $R_n(t)$ by 
\begin{equation}
\label{def:r_ik}
R_n(t) =: \sum_{i=1}^{p-1+s}\sum_{k=1}^{n} \frac{r_{i,k}}{(t+k)^i},
\end{equation}
where the coefficients $r_{i,k} \in \mathbb{Q}$ are uniquely determined by $R_n(t)$. 
For any integer $n > (p+s)^4$, we define the function $\widetilde{R}_n(t)$ by
\begin{equation}
\label{definition_widetildeR_n}
\widetilde{R}_n(t) := \sum_{k=1}^{n} r_{1,k}\log_p\langle t + k\rangle + \sum_{i=2}^{p-1+s}\sum_{k=1}^{n} \frac{r_{i,k}}{(1-i)(t+k)^{i-1}}.
\end{equation}
\end{definition}
	
Note that $r_{i,k} = r_{i,k,n}$ also depends on $n$. 
For convenience, we omit the subscript $n$. 
In the sequel, we will also omit the subscript $n$ in some other notations when there is no danger. 
	
\begin{lemma}
\label{lemma_primitive}
The function $\widetilde{R}_n(t)$ is a primitive function of $R_n(t)$ on $\mathbb{Q}_p \setminus \{-1,-2,\ldots,-n\}$\textup{;} that is,
\[ 
\widetilde{R}_n^{\prime}(t) = R_n(t) \quad\text{for any~} t \in \mathbb{Q}_p \setminus \{-1,-2,\ldots,-n\}. 
\]
Moreover, for any $j=1,2,\ldots,p-1$, the function
\[ 
\mathbb{Z}_p\to \mathbb{Q}_p, \quad t\mapsto \widetilde{R}_n\left(t+\frac{j}{p}\right),  
\]
is overconvergent, i.e. $\widetilde{R}_n\left(t+\frac{j}{p}\right)\in C^\dagger(\mathbb{Z}_p,\mathbb{Q}_p)$.
\end{lemma}
	
\begin{proof}
For $t \in \mathbb{Q}_p \setminus \{-1,-2,\ldots,-n\}$, we have
\[
\widetilde{R}_n^{\prime}(t) = \sum_{k=1}^{n} \frac{r_{1,k}}{t+k} + \sum_{i=2}^{p-1+s}\sum_{k=1}^{n} \frac{r_{i,k}}{(t+k)^i} = R_n(t). 
\]
		
For any $j \in \{1,2,\ldots, p-1\}$, any $k \in \{1,2,\ldots,n\}$ and any $t \in \mathbb{Z}_p$, we have 
\[ 
1 + \left(k+\frac{j}{p}\right)^{-1}t \in 1+p\mathbb{Z}_p.
\]
Therefore, we have
\[ 
\left\langle t+k+\frac{j}{p} \right\rangle = \left\langle 1 + \left(k+\frac{j}{p}\right)^{-1}t\right\rangle\left\langle k+\frac{j}{p}\right\rangle = \left(1 + \left(k+\frac{j}{p}\right)^{-1}t\right)\left\langle k+\frac{j}{p} \right\rangle. 
\] 
Thus,
\begin{align}
\widetilde{R}_n\left(t+\frac{j}{p}\right) = &\sum_{k=1}^{n} r_{1,k}\log_p\left\langle k+\frac{j}{p} \right\rangle + \sum_{k=1}^{n} r_{1,k}\log_p\left(1+\left(k+\frac{j}{p}\right)^{-1}t\right) \notag\\
&+ \sum_{i=2}^{p-1+s}\sum_{k=1}^{n} \frac{r_{i,k}}{(1-i)\left(k+\frac{j}{p}\right)^{i-1}}\left(1+\left(k+\frac{j}{p}\right)^{-1}t\right)^{1-i}, \quad t \in \mathbb{Z}_p .\label{Rn(t+theta)}
\end{align} 
Clearly, each summand on the right-hand side of \eqref{Rn(t+theta)} belongs to $C^\dagger(\mathbb{Z}_p,\mathbb{Q}_p)$. 
Therefore, $\widetilde{R}_n(t+\theta)$ is overconvergent. 
The proof of Lemma \ref{lemma_primitive} is complete.
\end{proof}

\section{Linear forms}
\label{sec:LinearForms}
	
In this section, we define for each $j=1,2,\ldots,p-1$ a linear form in $1$ and $p$-adic Hurwitz zeta values as a Volkenborn integral over $-\widetilde{R}_n\left(t+\frac{j}{p}\right)$.
Taking the sum of these linear forms, we obtain linear forms in $1$ and $p$-adic zeta values. 

\begin{lemma}
\label{rho_1_is_zero}
Define $\rho_1 := \sum_{k=1}^{n} r_{1,k}$. 
Then we have $\rho_1 = 0$.
\end{lemma}
	
\begin{proof}
By \eqref{def:r_ik} and $\deg R_n \leqslant -2$, we have 
\[ 
\rho_1 = \lim_{t \to \infty} tR_n(t) = 0. 
\]
\end{proof}
	
For any $j \in \{1,2,\ldots,p-1\}$, by \eqref{definition_widetildeR_n} we have
\[ 
\widetilde{R}_n\left(t+\frac{j}{p}\right) = \sum_{k=1}^{n} r_{1,k}\log_p\left\langle t+k+\frac{j}{p} \right\rangle + \sum_{i=2}^{p-1+s}\sum_{k=1}^{n} \frac{r_{i,k}}{(1-i)\left(t+k+\frac{j}{p}\right)^{i-1}}. 
\]
Note that each summand on the right-hand side above is Volkenborn integrable.
	
\begin{definition}\label{def:S_theta}
For any $j \in \{1,2,\ldots,p-1\}$, we define
\[ 
S_{j/p} := -\int_{\mathbb{Z}_p} \widetilde{R}_n\left(t+\frac{j}{p} \right) \mathrm{d}t. 
\]
\end{definition}
	
It turns out that $S_{j/p}$ is a linear form in $1$ and $p$-adic Hurwitz zeta values.
	
\begin{lemma}
\label{lemma_linear_form_S_theta}
For any $j \in \{1,2,\ldots,p-1\}$, we have
\[ 
S_{j/p} = \rho_{0,j/p} + \sum_{i=2}^{p-1+s} \rho_i \cdot \omega\left(\frac{j}{p}\right)^{1-i}\zeta_p\left(i,\frac{j}{p}\right),  
\]
where the coefficients
\begin{equation}
\label{definition_rho_i}
\rho_{i}:=\sum_{k=1}^{n} r_{i,k}, \quad (2 \leqslant i \leqslant p-1+s)
\end{equation}
do not depend on $j \in \{1,2,\ldots,p-1\}$, and 
\begin{equation}
\label{definition_rho_0theta}
\rho_{0,j/p}:=-\sum_{i=1}^{p-1+s}\sum_{k=1}^n\sum_{\nu=0}^{k-1} \frac{r_{i,k}}{\left(\nu+\frac{j}{p}\right)^{i}}.
\end{equation}
\end{lemma}
	
\begin{proof}
Fix $j \in \{1,2,\ldots,p-1\}$.	
We have 
\begin{equation}
\label{S_theta}
S_{j/p} = -\sum_{k=1}^{n} r_{1,k}\int_{\mathbb{Z}_p} \log_p\left\langle t+k+\frac{j}{p} \right\rangle \mathrm{d}t - \sum_{i=2}^{p-1+s}\sum_{k=1}^{n} r_{i,k}\int_{\mathbb{Z}_p} \frac{\mathrm{d}t}{(1-i)\left(t+k+\frac{j}{p}\right)^{i-1}}.
\end{equation}
By Lemma \ref{lem:Volkenborn_translation}, for any $i\in\{2,3,\ldots,p-1+s\}$ and $k \in \{1,2,\ldots,n\}$ we have
\begin{equation}
\label{int_log_p(t+k+theta)}
\int_{\mathbb{Z}_p} \log_p\left\langle t+k+\frac{j}{p} \right\rangle \mathrm{d}t = \int_{\mathbb{Z}_p} \log_p\left\langle t+\frac{j}{p} \right\rangle \mathrm{d}t + \sum_{\nu=0}^{k-1} \frac{1}{\nu +\frac{j}{p}}, 
\end{equation}
and using Eq. \eqref{eq:zeta_Volkenborn}, we get
\begin{align}
\int_{\mathbb{Z}_p}\frac{\mathrm{d}t}{(1-i)\left(t+k+\frac{j}{p}\right)^{i-1}} &= \int_{\mathbb{Z}_p}\frac{\mathrm{d}t}{(1-i)\left(t+\frac{j}{p}\right)^{i-1}} + \sum_{\nu=0}^{k-1} \frac{1}{\left(\nu+\frac{j}{p}\right)^{i}} \notag\\
&= -\omega\left(\frac{j}{p}\right)^{1-i}\zeta_p\left(i,\frac{j}{p}\right) + \sum_{\nu=0}^{k-1} \frac{1}{\left(\nu+\frac{j}{p}\right)^{i}}. \label{int_(t+k+theta)^(1-i)}
\end{align}
Substituting \eqref{int_log_p(t+k+theta)} and \eqref{int_(t+k+theta)^(1-i)} into \eqref{S_theta}, we obtain
\[ 
S_{j/p} = \rho_{0,j/p} + \sum_{i=2}^{p-1+s} \rho_i \cdot \omega\left(\frac{j}{p}\right)^{1-i}\zeta_p\left(i,\frac{j}{p}\right) - \rho_1 \cdot \int_{\mathbb{Z}_p} \log_p\left\langle t+\frac{j}{p} \right\rangle \mathrm{d}t. 
\]
Since $\rho_1 = 0$ by Lemma \ref{rho_1_is_zero}, the proof of Lemma \ref{lemma_linear_form_S_theta} is complete.
\end{proof}
	
\bigskip

Next, we take the sum of $S_{j/p}$ ($j=1,2,\ldots,p-1$) to obtain linear forms in $1$ and $p$-adic zeta values.
	
\begin{definition}
\label{definition_linear_form_S_n}
We define
\[ 
S_n := \sum_{j=1}^{p-1} S_{j/p}. 
\]
\end{definition}
	
\begin{lemma}
\label{lemma_linear_form_S_n}
We have
\[ 
S_n = \rho_{0} + \sum_{3 \leqslant i \leqslant p-1+s \atop i \text{~odd}} \rho_i \cdot p^{i} \zeta_p(i), 
\]
where the coefficients $\rho_i$ are those in \eqref{definition_rho_i} and 
\begin{equation}
\label{definition_rho_0}
\rho_{0} := \sum_{j=1}^{p-1} \rho_{0,j/p}.
\end{equation}
\end{lemma}
	
\begin{proof}
By Lemma \ref{lemma_linear_form_S_theta} and Lemma \ref{lem:padicHurwitz_zeta}, we have
\begin{align*}
S_n &= \sum_{j=1}^{p-1} \rho_{0,j/p} + \sum_{i=2}^{p-1+s} \rho_i \sum_{j=1}^{p-1} \omega\left(\frac{j}{p}\right)^{1-i}\zeta_p\left(i,\frac{j}{p}\right) \\
&= \rho_{0} + \sum_{i=2}^{p-1+s} \rho_i \cdot p^{i} \zeta_p(i).
\end{align*} 
Since $\zeta_p(2k)=0$ for every positive integer $k$, the proof of Lemma \ref{lemma_linear_form_S_n} is complete.
\end{proof}

\section{Arithmetic properties}
\label{sec:Arithmetic}
	
The goal of this section is to study the arithmetic properties of the coefficients of the linear forms constructed in Section \ref{sec:LinearForms}. 
	
As usual, we denote by $d_n := \operatorname{LCM}\{1,2,\ldots,n\}$ the least common multiple of the smallest $n$ positive integers. 
For any non-negative integer $\lambda$, we define the differential operator
\[
\mathcal{D}_{\lambda}:=\frac{1}{\lambda !} \left(\frac{\mathrm{d}}{\mathrm{d}t}\right)^{\lambda}.
\] 
	
We first state two basic lemmas.
	
\begin{lemma}[{\cite[Lemma 16]{Zud2004}}]
\label{lemma_G(t)}
Let $n$ be any non-negative integer and let $G(t)=n!/(t)_{n+1}$. 
Then we have
\[ 
d_n^{\lambda} \mathcal{D}_{\lambda}\left( G(t)(t+k) \right)\big|_{t=-k} \in \mathbb{Z} 
\]
for any integer $k \in \{0,1,\ldots,n\}$ and any integer $\lambda \geqslant 0$.
\end{lemma}
		
\begin{lemma}[{\cite[Lemma 4.2]{Lai2025+}}]
\label{lemma_F(t)}
Let $n$ be any non-negative integer. 
Let $a,b$ be any integers with $b>0$. 
Consider the polynomial
\[ 
F(t) = b^{n}\left(\prod_{q \mid b \atop q \text{~prime}} q^{\left\lfloor n/(q-1) \right\rfloor} \right)\cdot \frac{\left(t+ \frac{a}{b} \right)_{n}}{n!}.  
\]
Then, for any integer $k \in \mathbb{Z}$ and any integer $\lambda \geqslant 0$, we have
\[ 
d_n^{\lambda} \mathcal{D}_{\lambda} (F(t)) \big|_{t=-k} \in \mathbb{Z}. 
\]
\end{lemma}

Now, we bound the denominators of the coefficients $r_{i,k}$ defined by \eqref{def:r_ik}.
	
\begin{lemma}
\label{lemma_r_ik_weak}
For any $i \in \{1,2,\ldots,p-1+s\}$ and any $k \in \{1,2,\ldots,n\}$, we have
\[ 
d_n^{p-1+s-i} r_{i,k} \in \mathbb{Z}. 
\]
\end{lemma}
	
\begin{proof}
By \eqref{def:r_ik}, we have
\[
r_{i,k} = \mathcal{D}_{p-1+s-i}(R_n(t)(t+k)^{p-1+s})\big|_{t=-k}.
\]
Let
\begin{align*}
G(t) &= \frac{n!}{(t)_{n+1}},\\
F_{j/p}(t) &= p^n \cdot p^{\lfloor n/(p-1) \rfloor} \cdot \frac{\left(t+\frac{j}{p}\right)_{n}}{n!} \quad\text{for any~} j \in \{1,2,\ldots,p-1\},\\
F_{\star}(t) &= t^{M_0},\\
A &= p^{n-(p-1)\lfloor n/(p-1)\rfloor} \in \mathbb{N}. \notag
\end{align*}
It is straightforward to check that the following identity holds (see Definition \ref{def:R_n(t)}):
\begin{equation}
\label{Rn(t)_building_blocks}
R_n(t) =  A \cdot F_{\star}(t) \cdot G(t)^{p-1+s} \cdot \prod_{j=1}^{p-1} F_{j/p}(t).
\end{equation}
Therefore,
\begin{align*}
r_{i,k} &= \mathcal{D}_{p-1+s-i}(R_n(t)(t+k)^{p-1+s})\big|_{t=-k} \\
&= \mathcal{D}_{p-1+s-i}\left( A \cdot F_{\star}(t) \cdot (G(t)(t+k))^{p-1+s} \cdot \prod_{j=1}^{p-1} F_{j/p}(t) \right)\Big|_{t=-k}.
\end{align*}
Applying the Leibniz rule, we obtain
\begin{align}
d_n^{p-1+s-i}r_{i,k} = A \cdot \sum_{\boldsymbol{\lambda}}~  &d_n^{\lambda_{\star}} \mathcal{D}_{\lambda_{\star}}(F_{\star}(t)) \notag\\
&\times \prod_{\nu=1}^{p-1+s}d_n^{\lambda_\nu}\mathcal{D}_{\lambda_\nu}(G(t)(t+k)) \cdot \prod_{j=1}^{p-1} d_n^{\lambda_{j/p}}\mathcal{D}_{\lambda_{j/p}}(F_{j/p}(t)) \big|_{t=-k}, \label{eqn_r_ik}
\end{align}
where the sum is taken over all families of non-negative integers 
\[
\boldsymbol{\lambda} = (\lambda_{\star},(\lambda_\nu)_{1 \leqslant \nu \leqslant p-1+s}, (\lambda_{j/p})_{1 \leqslant j \leqslant p-1}) \] 
such that
\[ 
\lambda_{\star}+\sum_{\nu=1}^{p-1+s} \lambda_{\nu} + \sum_{j=1}^{p-1} \lambda_{j/p} = p-1+s-i. 
\]
By Lemma \ref{lemma_G(t)} and Lemma \ref{lemma_F(t)}, the value at $t=-k$ of each factor in the product in \eqref{eqn_r_ik} is an integer. 
Clearly, $\mathcal{D}_{\lambda_{\star}}(F_{\star}(t)) \big|_{t=-k}$ is also an integer. 
Therefore, we have $d_n^{p-1+s-i}r_{i,k} \in \mathbb{Z}$ as desired. 
\end{proof}
	
\bigskip
	
By Lemma \ref{lemma_r_ik_weak}, all of $d_n^{p-1+s-i}r_{i,k}$ are integers. 
The next lemma shows that these integers $d_n^{p-1+s-i}r_{i,k}$ have a large common divisor. 
A similar phenomenon is a key ingredient in the proof of Zudilin's theorem \cite{Zud2001}. 
Recall that we have assumed $n>(p+s)^4$.
	
\begin{lemma}
\label{lemma_r_ik}
For any $i \in \{1,2,\ldots,p-1+s\}$ and any $k \in \{1,2,\ldots,n\}$, we have
\[ 
\Phi_{n}^{-1} d_n^{p-1+s-i} r_{i,k} \in \mathbb{Z}, 
\]
where
\begin{align}
\Phi_{n} := \prod_{\sqrt{pn} < q \leqslant n \atop q~\text{prime}} q^{\phi\left( \frac{n}{q} \right)}, \label{def:Phi_n}
\end{align}
with the function $\phi:~\mathbb{R} \rightarrow \mathbb{Z}$ defined by
\begin{equation}
\label{def:phi}
\phi(x) := 
\begin{cases}
0, &\text{~if~} \{x\} \in \left[0,\frac{2}{p}\right), \\
j-1, &\text{~if~} \{x\} \in 	\left[\frac{j}{p},\frac{j+1}{p}\right), ~j=2,3,\ldots,p-1.
\end{cases}
\end{equation}
Here, $\{x\} = x -\lfloor x \rfloor$ denotes the fractional part of a real number $x$.
\end{lemma}
	
\begin{proof}
Replacing $i$ by $p-1+s-i$, we need to prove that
\begin{equation*}
\Phi_{n}^{-1} d_n^{i} r_{p-1+s-i,k} \in \mathbb{Z}, \quad i \in \{0,1,\ldots,p-2+s\}, ~k \in \{1,2,\ldots,n\}. 
\end{equation*}
By Lemma \ref{lemma_r_ik_weak}, we have 
\[ 
d_n^{i} r_{p-1+s-i,k} \in \mathbb{Z}. 
\]
It remains to prove that, for any prime $q$ such that $\sqrt{pn} < q \leqslant n$ we have
\begin{equation}
\label{eqn:lemma_r_ik}
v_q(r_{p-1+s-i,k}) \geqslant \phi\left( \frac{n}{q} \right) - i, \quad i \in \{0,1,\ldots,p-2+s\}.
\end{equation}
Fix such a prime $q$ and fix $k \in \{1,2,\ldots,n\}$. We prove \eqref{eqn:lemma_r_ik} by induction on $i$.
		
We first consider the case $i=0$ for \eqref{eqn:lemma_r_ik}. 
Note that we can rewrite $R_n(t)$ (see Definition \ref{def:R_n(t)}) as
\begin{align}
R_n(t) = n!^s \cdot t^{M_0} \cdot \frac{p^{-1} \cdot (pt)_{pn+1}}{(t)_{n+1}^{p+s}}. \label{rewrite_R_n(t)}
\end{align}
We have 
\begin{align*}
r_{p-1+s,k} &= R_n(t)(t+k)^{p-1+s} |_{t=-k} \\
&= \pm k^{M_0} \cdot \binom{n}{k}^s \cdot  \frac{(pk)!(pn-pk)!}{k!^{p}(n-k)!^p}. 
\end{align*}
Since $q > \sqrt{pn}$, we have 
\begin{align*}
v_q(r_{p-1+s,k}) &\geqslant v_q\left( \frac{(pk)!(pn-pk)!}{k!^{p}(n-k)!^p} \right) \\
&= \left\lfloor \frac{pk}{q} \right\rfloor + \left\lfloor \frac{pn-pk}{q} \right\rfloor - p\left\lfloor \frac{k}{q} \right\rfloor - p \left\lfloor \frac{n-k}{q} \right\rfloor.
\end{align*}
It is elementary to check that, for any real numbers $x,y$, we have
\[ 
\left\lfloor py \right\rfloor + \left\lfloor px-py \right\rfloor - p\left\lfloor y \right\rfloor - p \left\lfloor x-y \right\rfloor \geqslant \phi(x), 
\]
where $\phi(x)$ is given by \eqref{def:phi}. 
(We refer the reader to \cite[Fig. 9 and Fig. 10, p. 529]{Zud2002} or \cite[Fig. 1, p. 24]{Lai2025+} for the method of estimating a finite sum of floor functions). 
Taking $x=n/q$ and $y=k/q$, we obtain
\[ 
v_q(r_{p-1+s,k}) \geqslant \phi\left( \frac{n}{q} \right).  
\]
The case $i=0$ of \eqref{eqn:lemma_r_ik} is proved. 
		
Suppose now $i \in \{1,2,\ldots,p-2+s\}$ and \eqref{eqn:lemma_r_ik} is proved for any smaller $i$. 
Define
\begin{align}
U(t) &:= \frac{\mathcal{D}_1\left( R_n(t)(t+k)^{p-1+s} \right)}{R_n(t)(t+k)^{p-1+s}} \notag \\
&= \frac{M_0}{t}+\sum_{\nu = 0 \atop \nu \neq pk}^{pn} \frac{1}{t + \frac{\nu}{p}} -(p+s)\sum_{\nu=0 \atop \nu \neq k}^{n}\frac{1}{t+\nu}. \label{def:U(t)}
\end{align}
Then the Leibniz rule implies
\begin{align}
r_{p-1+s-i,k} &= \mathcal{D}_{i}\left( R_n(t)(t+k)^{p-1+s} \right)|_{t=-k} \notag\\
&= \frac{1}{i} \mathcal{D}_{i-1}\left( U(t) \cdot R_n(t)(t+k)^{p-1+s} \right)|_{t=-k} \notag\\
&= \frac{1}{i} \sum_{m = 0}^{i-1} \mathcal{D}_{m}\left(U(t)\right) \mathcal{D}_{i-1-m}\left( R_n(t)(t+k)^{p-1+s} \right)|_{t=-k} \notag\\
&= \frac{1}{i} \sum_{m = 0}^{i-1} r_{p-1+s-(i-1-m),k} \mathcal{D}_{m}\left(U(t)\right)|_{t=-k}. \label{5.12}
\end{align}
By \eqref{def:U(t)}, we have
\begin{align*}
&\mathcal{D}_{m}\left(U(t)\right)|_{t=-k} \\
&= (-1)^{m}\left( \frac{M_0}{(-k)^{m+1}} + \sum_{\nu = 0 \atop \nu \neq pk}^{pn} \frac{1}{\left(-k + \frac{\nu}{p}\right)^{m+1}} -(p+s)\sum_{\nu=0 \atop \nu \neq k}^{n}\frac{1}{(-k+\nu)^{m+1}} \right). 
\end{align*}
Since $q > \sqrt{pn}$, each term on the right-hand side above has $q$-adic order at least $-(m+1)$. 
Thus, we have
\begin{equation}
\label{5.13}
v_q\left( \mathcal{D}_{m}\left(U(t)\right)|_{t=-k} \right) \geqslant -m-1.
\end{equation}
On the other hand, by the induction hypothesis, we have
\begin{equation}
\label{5.14}
v_q(r_{p-1+s-(i-1-m),k}) \geqslant \phi\left( \frac{n}{q} \right) - (i-1-m).
\end{equation}
Noting that $q > \sqrt{pn} > p+s > i$ (since $n>(p+s)^4$), we conclude from \eqref{5.12}, \eqref{5.13} and \eqref{5.14} that
\[ 
v_q(r_{p-1+s-i,k}) \geqslant \phi\left( \frac{n}{q} \right) - i, 
\]
which completes the induction procedure. 
The proof of Lemma \ref{lemma_r_ik} is complete.
\end{proof}
	
\bigskip
	
Next, we bound the denominators of $\rho_i$ ($i \in \{0\}\cup\{2,3,\ldots,p-1+s\}$), which are the coefficients of the linear form $S_n$ (see Lemma \ref{lemma_linear_form_S_n}). 
	
\begin{lemma}
\label{lemma_rho_i}
For any $i \in \{2,3,\ldots,p-1+s\}$, we have
\[ 
\Phi_{n}^{-1} d_n^{p-1+s-i} \rho_{i} \in \mathbb{Z}. 
\]
\end{lemma} 
	
\begin{proof}
It follows immediately from Eq. \eqref{definition_rho_i} and Lemma \ref{lemma_r_ik}.
\end{proof}
	
\medskip
	
It is more involved to bound the denominator of $\rho_0$. 
Recall that $\rho_{0} = \sum_{j=1}^{p-1} \rho_{0,j/p}$ (see \eqref{definition_rho_0}) and $\rho_{0,j/p}$ is defined by \eqref{definition_rho_0theta}. 
We first bound the denominators of $\rho_{0,j/p}$. 
	
\begin{lemma}
\label{lemma_rho_0theta_general}
For any $j \in \{1,2,\ldots,p-1\}$, we have
\[ 
\Phi_{n}^{-1} d_n^{p-1+s} \rho_{0,j/p} \in \mathbb{Z}. 
\]
\end{lemma}
	
\begin{proof}
By \eqref{definition_rho_0theta}, we have
\[ 	
\Phi_{n}^{-1} d_n^{p-1+s}\rho_{0,j/p} =-\sum_{k=1}^n\sum_{\nu=0}^{k-1} \left(\sum_{i=1}^{p-1+s}\frac{\Phi_{n}^{-1} d_n^{p-1+s}r_{i,k}}{\left(\nu+\frac{j}{p}\right)^{i}} \right). 
\]
It suffices to prove that for any pair of integers $(k,\nu)$ such that $0 \leqslant \nu < k \leqslant n$, the most inner sum above is an integer. 
		
We argue by contradiction. 
Suppose that there is a pair $(k_0,\nu_0)$ with $0 \leqslant \nu_0 < k_0 \leqslant n$ such that 
\[ 
\sum_{i=1}^{p-1+s}\frac{\Phi_{n}^{-1} d_n^{p-1+s}r_{i,k_0}}{\left(\nu_0+\frac{j}{p}\right)^{i}} \not\in \mathbb{Z}. 
\]
Note that the linear polynomial $t+k_0-\nu_0-1+(p-j)/p$ is a factor of $(t+(p-j)/p)_n$, and hence a factor of $R_n(t)$. 
Thus,
\[ 
R_n\left( -k_0+\nu_0+1 - \frac{p-j}{p} \right) = 0. 
\]
It follows from \eqref{def:r_ik} that
\[ 
-\sum_{i=1}^{p-1+s}\sum_{k=1\atop k \neq k_0 }^{n} \frac{\Phi_{n}^{-1} d_n^{p-1+s}r_{i,k}}{\left( -k_0+\nu_0+k+\frac{j}{p} \right)^i}= \sum_{i=1}^{p-1+s}\frac{\Phi_{n}^{-1} d_n^{p-1+s}r_{i,k_0}}{\left(\nu_0+\frac{j}{p}\right)^{i}} \not\in \mathbb{Z}. 
\]
Therefore, there exist a prime number $q$ and integers $i_0,i_1,k_1$ satisfying $i_0,i_1 \in \{1,2,\ldots,p-1+s\}$, $k_1 \in \{1,2,\ldots,n\}$, $k_1 \neq k_0$, and
\[ 
v_q\left( \frac{\Phi_{n}^{-1} d_n^{p-1+s}r_{i_1,k_1}}{\left( -k_0+\nu_0+k_1+\frac{j}{p} \right)^{i_1}} \right) < 0, \quad v_q\left( \frac{\Phi_{n}^{-1} d_n^{p-1+s}r_{i_0,k_0}}{\left(\nu_0+\frac{j}{p}\right)^{i_0}} \right) < 0. 
\]
On the other hand, by Lemma \ref{lemma_r_ik} we have
\[ 
v_q\left( \Phi_{n}^{-1} d_n^{p-1+s-i_1}r_{i_1,k_1} \right) \geqslant 0, \quad v_q\left( \Phi_{n}^{-1} d_n^{p-1+s-i_0}r_{i_0,k_0} \right) \geqslant 0. 
\]
It follows that
\[ 
v_q\left( -k_0+\nu_0+k_1+\frac{j}{p} \right) > v_q(d_n), \quad v_q\left(\nu_0+\frac{j}{p}\right) > v_q(d_n), 
\]
and hence
\[ 
v_q(k_1-k_0) > v_q(d_n). 
\]
But this contradicts $0 < |k_1-k_0| \leqslant n$. 
The proof of Lemma \ref{lemma_rho_0theta_general} is complete.
\end{proof}
	
\medskip

\begin{lemma}
\label{lemma_integrality_rho_0}
We have
\[ 
\Phi_{n}^{-1} d_n^{p-1+s} \rho_{0} \in \mathbb{Z} 
\]
\end{lemma}
	
\begin{proof}
It follows immediately from Eq. \eqref{definition_rho_0} and Lemma \ref{lemma_rho_0theta_general}.
\end{proof}

\section{$p$-adic norm}
\label{sec:p-adic_norm}

In this section, we bound the $p$-adic norm of the linear forms constructed in Section \ref{sec:LinearForms}. 
	
\begin{lemma}
\label{lemma_u_k}
For any $j \in \{1,2,\ldots,p-1\}$, we have $R_n(t+j/p) \in C^\dagger(\mathbb{Z}_p,\mathbb{Q}_p)$ and the following estimate for the $p$-adic order of $\mathcal{L}_1(R_n(t+j/p))$: 
\[
v_p\left( \mathcal{L}_1\left( R_n\left( t+\frac{j}{p} \right) \right) \right) \geqslant (p+1+s)(n+1) + s \cdot v_p(n!) -M_0 - 3 - \left\lfloor \frac{\log (n+1)}{\log p} \right\rfloor.
\]
\end{lemma}
	
\begin{proof}
By Lemma \ref{lemma_primitive}, we have $\widetilde{R}_n(t+j/p) \in C^\dagger(\mathbb{Z}_p,\mathbb{Q}_p)$. 
It follows that $R_n(t+j/p) \in C^\dagger(\mathbb{Z}_p,\mathbb{Q}_p)$ because $C^\dagger(\mathbb{Z}_p,\mathbb{Q}_p)$ is closed under taking derivative.
		
By \eqref{def:R_n(t)}, we can rewrite $R_n(t)$ as
\[ 
R_n(t) = p^{pn} \cdot n!^s \cdot t^{M_0} \cdot \frac{\prod_{\nu=1}^{p-1} \prod_{m=0}^{n-1} \left( t+m+\frac{\nu}{p} \right)}{(t)_{n+1}^{p-1+s}}, 
\]
Thus, we have
\begin{align}
R_n\left(t+\frac{j}{p}\right) = &~p^{(p+1+s)(n+1)-M_0-2} \cdot n!^s \label{constant_term}\\
&\times  \prod_{m =0}^{n-1} (t+m+1) \label{product_term_1} \\
&\times \left( pt+j \right)^{M_0} \cdot \prod_{1 \leqslant \nu \leqslant p-1\atop \nu \neq p-j}\prod_{m =0}^{n-1} (pt+pm+j+\nu) \label{product_term_2}\\
&\times \prod_{m=0}^{n}\left( pt + pm+ j \right)^{-(p-1+s)}. \label{product_term_3}
\end{align}
Clearly, the factors in the lines \eqref{product_term_2} and \eqref{product_term_3} belong to $\mathbb{Z}_p\llbracket pt \rrbracket$. 
Thus, we have
\begin{equation}\label{expression_for_R_n(t+theta)}
R_n\left(t+\frac{j}{p}\right) = p^{(p+1+s)(n+1)-M_0-2} \cdot n!^s \cdot P(t) \cdot Q(t), 
\end{equation}
where $P(t) \in \mathbb{Z}[t]$ is the polynomial in the line \eqref{product_term_1}, and $Q(t) \in \mathbb{Z}_p\llbracket pt \rrbracket$. 
Let us write
\[ 
R_n\left(t+\frac{j}{p}\right) = \sum_{k=0}^{\infty} u_k t^{k}. 
\]
Noting that $\deg P = n$, we deduce from \eqref{expression_for_R_n(t+theta)} that $u_k \in \mathbb{Z}_p$ and 
\begin{equation*}
v_p(u_k) \geqslant (p+1+s)(n+1) + s\cdot v_p(n!) -M_0 -2 + \max\{ 0,k-n \} 
\end{equation*}
for any $k \geqslant 0$. 
Therefore, 
\begin{align*}
&v_p\left(\mathcal{L}_1\left( R_n\left( t+\frac{j}{p} \right) \right) \right) = v_p\left( \sum_{k=0}^{\infty} u_k \cdot \frac{B_{k+1}}{k+1} \right) \\
\geqslant& (p+1+s)(n+1) + s\cdot v_p(n!) -M_0 -2 + \min_{k \geqslant 0} \left( \max\{0,k-n\} - v_p(k+1) -1 \right) \\
\geqslant&  (p+1+s)(n+1) + s\cdot v_p(n!) -M_0 -3 - \left\lfloor \frac{\log (n+1)}{\log p} \right\rfloor.
\end{align*}
The proof of Lemma \ref{lemma_u_k} is complete.
\end{proof}

\bigskip

Recall that $N_0=v_p(p-1+s)$ and $M_0=p^{2+N_0}\cdot s-1$. 
Note that there are infinitely many positive integers $N$ with $p^N\equiv p^{2+N_0} \pmod{p-1+s}$. 
For any such $N$ we have
\[
p^N(p-1)+M_0+1\equiv 0 \pmod{p-1+s}
\]
and
\begin{equation}
\label{def:n(N)}
n(N):=\frac{p^N(p-1)+M_0+1}{p-1+s}-1
\end{equation}
is a well-defined positive integer. 
Furthermore, note that $n(N)\to \infty$ as $N\to\infty $. 
Thus, the following set $I$ is an unbounded subset of $\mathbb{N}$:
\begin{equation}
\label{def:I}
I := \left\{  n \in \mathbb{N} ~\mid~ n>(p+s)^4 \text{~and~} n = n(N) \text{~for some~} N \in \mathbb{N}\right\}.
\end{equation}
The key point of the above definitions is that
\[
(n(N)+1)(p-1+s)-M_0=p^N(p-1)+1.
\]

\begin{lemma}
\label{lem:vpRntilde}
For any $n = n(N) \in I$, we have
\[
v_p\left( \sum_{j=1}^{p-1} \widetilde{R}_{n}\left(\frac{j}{p}\right) \right)=(p+s)(n+1)+s\cdot v_p(n!)-M_0-2-N.
\]
\end{lemma}

\begin{proof}
We have
\begin{align}
R_n\left(\frac{t}{p}\right) = &~p^{(p+s)(n+1)-M_0-1} n!^s \cdot \frac{1}{t^{(p-1+s)(n+1)-M_0}} \label{C} \\
&\times \prod_{\nu=1}^{p-1} \prod_{m =0}^{n-1} \left( t+pm+\nu \right) \label{P(t)} \\
&\times \prod_{m=0}^{n} \left( 1+ \frac{pm}{t} \right)^{-(p-1+s)}. \label{Q(t)} 
\end{align}
Let us denote by $P(t)\in \mathbb{Z}_p[t]$ the polynomial in \eqref{P(t)} and by $Q(t)\in \mathbb{Z}_p\llbracket p\cdot t^{-1}\rrbracket$ the power series in \eqref{Q(t)}. 
Note that $P(t)$ is of degree $(p-1)n$ and
\begin{equation}\label{eq:congP}
P(t)\equiv (t^{p-1}-1)^n \pmod {p\mathbb{Z}_p[t]}.
\end{equation}
Let us write
\[
R_n\left(\frac{t}{p}\right)=p^{(p+s)(n+1)-M_0-1}n!^s \cdot \sum_{k=p-1+(n+1)s-M_0}^\infty h_k t^{-k}.
\]
Since $Q(t)\in \mathbb{Z}_p\llbracket p\cdot t^{-1}\rrbracket$, we get
\begin{equation}
\label{eq:vphk}
v_p(h_k) \geqslant \max\{0,k-(p-1+s)(n+1)+M_0\}.
\end{equation}
Equations \eqref{eq:congP}, \eqref{C} and $(n+1)(p-1+s)-M_0=p^N(p-1)+1$ imply that
\begin{align}
h_{k}&\equiv 0 \pmod {p}, \quad \text{ for $0\leqslant k\leqslant p^N(p-1)+1$ with $k\not\equiv 1\mod (p-1)$}, \label{eq:cong_hk1} \\
h_{ p^N(p-1)+1}&\equiv 1 \pmod {p}.  \label{eq:cong_hk2} 
\end{align}
We can now compute $\widetilde{R}_n(t/p)$:
\[
\widetilde{R}_n\left(\frac{t}{p}\right)=p^{(p+s)(n+1)-M_0-2}n!^s \cdot \sum_{k=p-1+(n+1)s-M_0}^\infty \frac{h_k}{1-k} t^{-(k-1)}.
\]
For $k>(n+1)(s+p-1)-M_0=p^N(p-1)+1$, we have by \eqref{eq:vphk}
\[
v_p\left( \frac{h_k}{1-k} \right)\geqslant 0.
\]
On the other hand, for $p-1+(n+1)s-M_0 \leqslant k \leqslant p^N(p-1)+1$, we have
\[
v_p\left(\frac{h_k}{1-k}\right) \geqslant -N.
\]
By \eqref{eq:cong_hk1} and \eqref{eq:cong_hk2}, we have equality if and only if $k=(p-1)p^N+1$. 
In fact, $k=(p-1)p^N+1$ is the only integer between $p-1+(n+1)s-M_0$ and $p^N(p-1)+1$ which is $\equiv 1\mod p^N$ and $\equiv 1 \mod (p-1)$. 
We can summarize the above discussion as follows
\begin{align*}
v_p\left(\frac{h_k}{1-k}\right)
\begin{cases}
>-N & \text{for}\ k\neq p^N(p-1)+1,\\
=-N & \text{for}\ k=p^N(p-1)+1.
\end{cases}
\end{align*}
For $m:=(p+s)(n+1)-M_0-2+s\cdot v_p(n!)-N$, we obtain
\begin{align*}
\sum_{j=1}^{p-1}\widetilde{R}_n\left(\frac{j}{p}\right)&\equiv p^{(p+s)(n+1)-M_0-2} n!^s\sum_{j=1}^{p-1} \frac{h_{p^N(p-1)+1}}{1-(p^N(p-1)+1)} j^{-p^N(p-1)} \mod p^{m+1}\\
&\equiv p^{(p+s)(n+1)-M_0-2-N} n!^s\sum_{j=1}^{p-1} h_{p^N(p-1)+1} j^{-p^N(p-1)} \mod p^{m+1}\\
&\equiv  p^{(n+1)(p+s)-2-M_0-N} n!^s\sum_{j=1}^{p-1} 1\equiv -  p^{(n+1)(p+s)-M_0-2-N} n!^s \mod p^{m+1}.
\end{align*}
This shows
\[
v_p\left( \sum_{j=1}^{p-1} \widetilde{R}_{n}\left(\frac{j}{p}\right) \right)=m=(p+s)(n+1)-M_0-2+s\cdot v_p(n!)-N,
\]
as desired.
\end{proof}

\bigskip

\begin{lemma}
\label{lemma_p_adic_norm_S_n}
For $n= n(N) \in I$, where $n(N)$ is defined by \eqref{def:n(N)} and $I$ is defined by \eqref{def:I}, we have
\[
v_p\left( S_n \right)=(p+s)(n+1)-M_0-2+s\cdot v_p(n!) -N.
\]
In particular, we have $S_n \neq 0$ and 
\[  
|S_n|_p = p^{-(p+ps/(p-1))n + o(n)}
\]
as $n \in I$ and $n \to \infty$.
\end{lemma}

\begin{proof}
By Definition \ref{def:S_theta} and Lemma \ref{lem:comparison_Volkenborn_Bernoulli}, we have
\[ 
S_{j/p} = - \mathcal{L}_1 \left( R_n\left( t+\frac{j}{p} \right) \right) - \widetilde{R}_n\left( \frac{j}{p} \right) 
\]
for any $j \in \{1,2,\ldots,p-1\}$. 
Then, by Definition \ref{definition_linear_form_S_n} , we have
\[
S_n = - \sum_{j=1}^{p-1} \mathcal{L}_1 \left( R_n\left( t+\frac{j}{p} \right) \right) - \sum_{j=1}^{p-1} \widetilde{R}_n\left( \frac{j}{p} \right).
\]

For $n = n(N) \in I$, by Lemma \ref{lemma_u_k} and Lemma \ref{lem:vpRntilde}, we have
\[
v_p\left(\sum_{j=1}^{p-1} \widetilde{R}_n\left(\frac{j}{p}\right)\right)<v_p\left(\mathcal{L}_1\left( R_n\left(t+\frac{j}{p}\right) \right)\right)
\]
for any $j \in \{1,2,\ldots,p-1\}$, and hence
\[
v_p\left( S_n \right) = v_p\left( \sum_{j=1}^{p-1} \widetilde{R}_{n}\left(\frac{j}{p}\right) \right)=(p+s)(n+1)-M_0-2+s\cdot v_p(n!)-N.
\]
In particular, $S_n \neq 0$ for $n \in I$. 
Finally, noting that $N = O(\log n)$ and 
\[
v_p(n!) = \frac{n}{p-1} + O(\log n),
\]
we obtain 
\[  
|S_n|_p = p^{-(p+ps/(p-1))n + o(n)}
\]
as $n \in I$ and $n \to \infty$.
\end{proof}

\section{Archimedean properties}
\label{sec:Archimedean}
	
In the following, we will estimate the Archimedean growth of the coefficients $\rho_{i}$ of the linear forms in $1$ and $p$-adic zeta values as $n\to \infty$. 
We will also estimate the Archimedean growth of $\Phi_n$ appeared in Lemma \ref{lemma_r_ik}. 
		
\begin{lemma}
\label{lemma_rho_i_estimate}
We have
\[  
\max_{0 \leqslant i \leqslant p-1+s} |\rho_{i}| \leqslant  2^{sn} \cdot  p^{pn+o(n)} \quad \text{as~} n \to \infty. 
\]
\end{lemma}
	
\begin{proof}
We have $\rho_1 = 0$ by Lemma \ref{rho_1_is_zero}. 
On the other hand, by \eqref{definition_rho_i}, \eqref{definition_rho_0theta} and \eqref{definition_rho_0}, we have
\[ 
|\rho_i| \leqslant \sum_{k=1}^{n} |r_{i,k}|\quad \text{for~} i \in \{2,3,\ldots,p-1+s\}, 
\]
and
\begin{align*}
|\rho_0| &\leqslant \sum_{j=1}^{p-1} |\rho_{0,j/p}| \\
& \leqslant \sum_{j=1}^{p-1}\sum_{i=1}^{p-1+s}\sum_{k=1}^n\sum_{\nu=0}^{k-1} \frac{|r_{i,k}|}{\left|\nu+\frac{j}{p}\right|^{i}} \\
& \leqslant (p-1+s)^2n^2p^{p-1+s} \cdot \max_{1 \leqslant i \leqslant p-1+s \atop 1 \leqslant k \leqslant n} |r_{i,k}|.
\end{align*}
It suffices to prove that
\begin{equation}
\label{r_ik_Archi}
\max_{1 \leqslant i \leqslant p-1+s \atop 1 \leqslant k \leqslant n} |r_{i,k}| \leqslant 2^{sn} \cdot p^{pn+o(n)} \quad \text{as~} n \to \infty.
\end{equation}
		
Now, fix any $i \in \{1,2,\ldots,p-1+s\}$ and any $k \in \{1,2,\ldots,n\}$. 
By \eqref{def:r_ik} and Cauchy's integral formula, we have
\[ 
r_{i,k} = \frac{1}{2\pi \sqrt{-1}} \int_{|z+k|=\frac{1}{2p}} (z+k)^{i-1}R_n(z) \mathrm{d}z, 
\]
and hence
\[ 
|r_{i,k}| \leqslant \sup_{z \in \mathbb{C}:~|z+k|=\frac{1}{2p}} |R_n(z)|. 
\]
		
By \eqref{rewrite_R_n(t)}, we have
\[ 
R_n(z) = n!^s \cdot z^{M_0}\cdot \frac{p^{-1} \cdot (pz)_{pn+1}}{(z)_{n+1}^{p+s}}. 
\]
Using triangle inequality, it is easy to show that for any complex number $z$ with $|z+k|=1/(2p)$, we have
\begin{align*}
|z^{M_0}| &\leqslant (n+1)^{M_0},\\
|(pz)_{pn+1}| &\leqslant (pk+1)!(pn-pk+1)! < (pk)! (pn-pk)! \cdot (2pn)^2, \\
|(z)_{n+1}| &\geqslant \left( \frac{1}{2p} \right)^3 \cdot \frac{k!(n-k)!}{n^2} > \frac{k!(n-k)!}{(2pn)^3}
\end{align*}
Therefore, we have
\begin{align*}
\sup_{z \in \mathbb{C}:~|z+k|=\frac{1}{2p}} |R_n(z)| &\leqslant \binom{n}{k}^s \cdot \frac{(pk)!}{k!^{p}} \cdot \frac{(pn-pk)!}{(n-k)!^p} \cdot (2pn)^{3p+3s+M_0+2} \\
&\leqslant 2^{sn} \cdot p^{pk} \cdot p^{pn-pk} \cdot (2pn)^{3p+3s+M_0+2}.
\end{align*}
It follows that \eqref{r_ik_Archi} holds. 
The proof of Lemma \ref{lemma_rho_i_estimate} is complete.
\end{proof}
	
\bigskip
	
It is well known that the prime number theorem implies
\begin{equation}
\label{d_n_Archi}
d_n = e^{n+o(n)} \quad \text{as~} n \to \infty.
\end{equation}
As usual, we denote the Euler–Mascheroni constant by $\gamma$. 
Let
\begin{equation}
\label{def_digamma}
\psi(x) = \frac{\Gamma'(x)}{\Gamma(x)} = -\gamma + \sum_{m=0}^{\infty}\left( \frac{1}{m+1} - \frac{1}{m+x} \right)
\end{equation}
be the digamma function. 
The following lemma is also a corollary of the prime number theorem (cf. \cite[Lemma 4.4]{Zud2002}).
\begin{lemma}
\label{lemma_Phi_n}
The limiting relation
\[ 
\lim_{n \to \infty} \frac{1}{n} \sum_{\sqrt{Cn} < q \leqslant n \atop \{ n/q \} \in [u,v) } \log q =  \psi(v) - \psi(u) + \frac{1}{v} - \frac{1}{u}  
\]
holds for any constant $C \geqslant 0$ and any interval $[u,v) \subset (0,1)$. 
Here $q$ runs only along primes and $\{n/q\}$ is the fractional part of $n/q$.
\end{lemma}
	
Next, we estimate the Archimedean growth of $\Phi_n$ appeared in Lemma \ref{lemma_r_ik}. 
	
\begin{lemma}
\label{Phi_n_Archi}
We have 
\begin{align*}
\Phi_n &= e^{\varpi_pn + o(n)} \quad \text{as~} n \to \infty,
\end{align*}
where 
\[ 
\varpi_p = \psi\left( \frac{1}{p} \right)  + 2p-1 + \gamma + p\left( \log p - \sum_{j=1}^{p} \frac{1}{j} \right). 
\]
\end{lemma}
	
\begin{proof}
By \eqref{def:Phi_n}, we have
\[ 
\Phi_{n} = \prod_{\sqrt{pn} < q \leqslant n \atop q~\text{prime}} q^{\phi\left( \frac{n}{q} \right)}, 
\]
where $\phi(x)$ is the piecewise constant function defined in \eqref{def:phi}. 
Then Lemma \ref{lemma_Phi_n} implies that
\begin{align*}
\lim_{n \to \infty} \frac{\log \Phi_n}{n} &=  \sum_{j=2}^{p-1} (j-1)\left( \psi\left( \frac{j+1}{p} \right) - \psi\left( \frac{j}{p} \right) + \frac{p}{j+1}  - \frac{p}{j} \right) \\
&= (p-2)\psi\left( 1 \right) +p-2 - \sum_{j=2}^{p-1} \psi\left( \frac{j}{p} \right) - p\sum_{j=2}^{p-1}\frac{1}{j} \\
&= \psi\left( \frac{1}{p} \right) + (p-1)\psi(1) +2p -1 - \sum_{j=1}^{p} \psi\left( \frac{j}{p} \right) - p\sum_{j=1}^{p}\frac{1}{j}.
\end{align*}
Using the following well-known identities (they follows easily from \eqref{def_digamma})
\begin{align*}
\psi(1) &= -\gamma, \\
\sum_{j=1}^{p} \psi\left( \frac{j}{p} \right) &= -p\left( \gamma +\log p \right),
\end{align*}
we obtain
\[ 
\lim_{n \to \infty} \frac{\log \Phi_n}{n} = \psi\left( \frac{1}{p} \right)  + 2p-1 + \gamma + p\left( \log p - \sum_{j=1}^{p} \frac{1}{j} \right) = \varpi_p. 
\]
In other words, 
\[ 
\Phi_n = e^{\varpi_p n +o(n)}. 
\]
\end{proof}
	
We have the following estimates for $\varpi_p$:
\begin{lemma}
\label{lem:estimate_varpi_p}
We have
\[
(p-1)(1-\gamma)-\log 2-1< \varpi_p < (p-1)(1-\gamma).
\]
\end{lemma}

\begin{proof}
By \cite[Cor. 3]{EGP2000} and $\psi(x+1)=\psi(x)+1/x$, we have for $x>0$ the inequality
\[
\log\left(x+\frac{1}{2}\right)- \frac{1}{x} < \psi(x) < \log\left(x+e^{-\gamma}\right)-\frac{1}{x}.
\]
In particular, we have for $x=1/p$ the estimates
\[
-\log 2 -p < \psi\left(\frac{1}{p}\right) <-p.
\]
Furthermore, we have
\[
\log p-\sum_{j=1}^p \frac{1}{j} < -\gamma < \log p -\sum_{j=1}^{p-1} \frac{1}{j}.
\]
Combining the last two estimates with the definition of $\varpi_p$ gives:
\[
-\log2-p +2p -1 +\gamma +p\left( -\gamma-\frac{1}{p} \right) < \varpi_p < -p +2p-1+\gamma+p(-\gamma).
\]
After simplifying both terms in the last inequality, we obtain the inequality in the statement of Lemma \ref{lem:estimate_varpi_p}.
\end{proof}

\section{Proof of the main theorem}
\label{sec:proof}
	
\begin{proof}[Proof of Theorem \ref{main_thm}]
For any integer $n$ in the set $I$ defined by \eqref{def:I}, we define
\[ 
\widehat{S}_n := \Phi_{n}^{-1} d_n^{p-1+s} \cdot S_n.  
\]
Then Lemma \ref{lemma_linear_form_S_n} implies
\[ 
\widehat{S}_n = \widehat{\rho}_{0} + \sum_{3 \leqslant i \leqslant p-1+s \atop i \text{~odd}} \widehat{\rho}_i \cdot p^{i} \zeta_p(i), 
\]
where
\[ 
\widehat{\rho}_{i} = \Phi_{n}^{-1} d_n^{p-1+s} \cdot \rho_i, \quad i \in \{ 0,1, 2,\ldots,p-1+s\}.
\]
By Lemma \ref{rho_1_is_zero}, Lemma \ref{lemma_rho_i} and Lemma \ref{lemma_integrality_rho_0}, we have $\widehat{\rho}_{i} \in \mathbb{Z}$ for all $i \in \{ 0,1,2,\ldots,p-1+s\}$.

As $n \in I$ and $n \to \infty$, by Eq. \eqref{d_n_Archi}, Lemma \ref{lemma_rho_i_estimate} and Lemma \ref{Phi_n_Archi}, we have
\[  
\max_{0 \leqslant i \leqslant p-1+s} |\widehat{\rho}_{i}| \leqslant 2^{sn} \cdot  p^{pn} \cdot e^{(p-1+s-\varpi_p)n + o(n)}.  
\]
On the other hand, Lemma \ref{lemma_p_adic_norm_S_n} implies 
\[ 
|\widehat{S}_{n}|_p = p^{-(p+ps/(p-1))n + o(n)}, 
\]
and $\widehat{S}_{n} \neq 0$ for $n \in I$.

Therefore, if the positive integer $s$ satisfies
\begin{equation}
\label{sss}
2^{s}p^{p}e^{p-1+s-\varpi_p}p^{-(p+ps/(p-1))} < 1,
\end{equation}
then 
\[ 
\max_{0 \leqslant i \leqslant p-1+s} |\widehat{\rho}_{i}| \cdot |\widehat{S}_{n}|_p  \to 0 \text{~as~} n \in I \text{~and~} n \to \infty,  
\]
and Lemma \ref{lem:irrationalityCrit} will imply that there exists an odd integer $i \in [3,p-1+s]$ such that $\zeta_p(i)$ is irrational. 
		
Take
\[ 
s = \left\lfloor \frac{p-1-\varpi_p}{\frac{p}{p-1}\log p -1 -\log 2}  \right\rfloor + 1, 
\]
which is the smallest integer satisfying \eqref{sss}. 
By Lemma \ref{lem:estimate_varpi_p} and $p\geqslant 5$ we know that this $s$ is positive. 
Then we conclude that there exists an odd integer $i \in [3,c_p]$ such that $\zeta_p(i)$ is irrational, where
\[ 
c_p = p + \frac{p-1-\varpi_p}{\frac{p}{p-1}\log p - 1 -\log 2}. 
\]
The proof of Theorem \ref{main_thm} is complete.
\end{proof}

\begin{remark}
\label{rem:bound_cp}
Using the estimates for $\varpi_p$ from Lemma \ref{lem:estimate_varpi_p}, it is not difficult to check that
\[
\text{(the greatest odd integer not exceeding $c_p$)} \leqslant p+\frac{p}{\log p} + 5
\]
for any $p\geqslant 5$.
Moreover, we have
\[
c_p =p + (\gamma+o(1))\cdot\frac{p}{\log p} \quad\text{as}\ p \to \infty,
\]
where $\gamma=0.577\ldots$ is the Euler–Mascheroni constant. 
\end{remark}	

\begin{acknowledgements}
The third author gratefully acknowledges the support through the DFG funded Collaborative Research Center SFB 1085 ’Higher Invariants’.
\end{acknowledgements}

\vspace*{3mm}
\begin{flushright}
\begin{minipage}{148mm}\sc\footnotesize
L.\,L., Beijing International Center for Mathematical Research, \\
Peking University, Beijing, China\\
{\it E--mail address}: {\tt lilaimath@gmail.com} \vspace*{3mm}
\end{minipage}
\end{flushright}
	
\begin{flushright}
\begin{minipage}{148mm}\sc\footnotesize
C.\,L., Beijing Institute of Mathematical Sciences and Applications (BIMSA) \&\\
Yau Mathematical Sciences Center (YMSC), Tsinghua University, Beijing, China \\
{\it E--mail address}: {\tt lupucezar@gmail.com,~lupucezar@bimsa.cn} \vspace*{3mm}
\end{minipage}
\end{flushright}
	
\begin{flushright}
\begin{minipage}{148mm}\sc\footnotesize
J.\,S., Department of Mathematics, University of Duisburg-Essen, Essen, Germany \\
{\it E--mail address}: {\tt johannes.sprang@uni-due.de } \vspace*{3mm}
\end{minipage}
\end{flushright}
	

\begin{thebibliography}{99}
		
\bibitem{Ape1979} R. Ap{\'e}ry, \textit{Irrationalit{\'e} de $\zeta(2)$ et $\zeta(3)$}, in \textit{Journ{\'e}es Arithm{\'e}tiques (Luminy, 1978)}, Ast{\'e}risque, vol. 61 (Soci{\'e}t{\'e} Math{\'e}matique de France, Paris, 1979), 11--13.

\bibitem{BR2001} K. Ball, T. Rivoal, \textit{Irrationalité d’une infinité de valeurs de la fonction zêta aux entiers impairs}, Invent. math. 146, 193--207 (2001).

\bibitem{Beu1987} F. Beukers, \textit{Irrationality proofs using modular forms}, Journ\'ees arithm\'etiques de Besan\c{c}on (Besan\c{c}on, 1985), Ast\'erisque No. 147--148 (1987), 271--283, 345.
        
\bibitem{Beu2008} F. Beukers, \textit{Irrationality of some $p$-adic $L$-values}, Acta Math. Sin. (Engl. Ser.) 24 (2008), no. 4, 663--686.

\bibitem{BZ2022+} F. Brown and W. Zudilin,
\textit{On cellular rational approximations to $\zeta(5)$}, Preprint \href{https://arxiv.org/abs/2210.03391v2}{{\tt arXiv:2210.03391v2 [math.NT]}} (2022).
		
\bibitem{Cal2005} F. Calegari, \textit{Irrationality of certain $p$-adic periods for small $p$}, Int. Math. Res. Not. (2005), no. 20, 1235--1249.
		
\bibitem{CDT2020+} F. Calegari, V. Dimitrov and Y. Tang, \textit{$p$-adic Eisenstein series, arithmetic holonomicity criteria, and irrationality of the $2$-adic period $\zeta_2(5)$}, \url{https://people.maths.ox.ac.uk/newton/lnts/VDimitrov-LNT.pdf}, 2020.
		
\bibitem{Coh2007} H. Cohen, \textit{Number Theory, Vol. \uppercase\expandafter{\romannumeral2}: Analytic and Modern Tools}, Grad. Texts in Math. 240, Springer, New York, 2007.
		
\bibitem{EGP2000} N. Elezovi\'c, C. Giordano and J. Pe\v{c}ari\'c, \textit{The best bounds in Gautschi's inequality}, Math. Inequal. Appl. 3(2) (2000), 239--252.

\bibitem{Fis2021+} S. Fischler, \textit{Linear independence of odd zeta values using Siegel's lemma}, Preprint \href{https://arxiv.org/abs/2109.10136}{{\tt arXiv:2109.10136 [math.NT]}} (2021).
        
\bibitem{FSZ2019} S. Fischler, J. Sprang and W. Zudilin, \textit{Many odd zeta values are irrational}, Compos. Math. 155(5) (2019), 938--952.

\bibitem{Lai2025} L. Lai, \textit{On the irrationality of certain $2$-adic zeta values}, Int. J. Number Theory 21 (2025), no. 1, 207--235.

\bibitem{Lai2025+} L. Lai, \textit{Small improvements on the Ball--Rivoal theorem and its $p$-adic variant}, Preprint \href{https://arxiv.org/abs/2407.14236v2}{{\tt arXiv:2407.14236v2 [math.NT]}} (2025).

\bibitem{Lai2025-II+} L. Lai, \textit{A note on the number of irrational odd zeta values, \uppercase\expandafter{\romannumeral2}}, Preprint \href{https://arxiv.org/abs/2501.05321}{{\tt arXiv:2501.05321 [math.NT]}} (2025).
    
\bibitem{LS2023+} L. Lai and J. Sprang, \textit{Many $p$-adic zeta values are irrational}, to appear in Michigan Math. J., 
Preprint \href{https://arxiv.org/abs/2306.10393v2}{{\tt arXiv:2306.10393v2 [math.NT]}} (2023).

\bibitem{LSZ2025+} L. Lai, J. Sprang and W. Zudilin, \textit{A note on the irrationality of $\zeta_2(5)$}, Preprint \href{https://arxiv.org/abs/2505.05005}{{\tt arXiv:2505.05005 [math.NT]}} (2025).
    
\bibitem{LY2020} L. Lai and P. Yu, \textit{A note on the number of irrational odd zeta values}, Compos. Math. 156 (2020), no. 8, 1699--1717.
		
\bibitem{Riv2000} T. Rivoal, \textit{La fonction z{\^e}ta de Riemann prend une infinit{\'e} de valeurs irrationnelles aux entiers impairs},  C. R. Acad. Sci. Paris S{\'e}r. I Math. 331 (2000), no. 4, 267--270.
		
\bibitem{Rob2000} A.\,M. Robert, \textit{A course in p-adic analysis}, Graduate Texts in Mathematics 198, Springer-Verlag, New York, 2000.
		
\bibitem{Sch2006} W.\,H. Schikhof, \textit{Ultrametric calculus}, Cambridge Studies in Advanced Mathematics, vol. 4, Cambridge University Press, Cambridge, 2006, Reprint of the 1984 original.
		
\bibitem{Spr2020} J. Sprang, \textit{Linear independence result for $p$-adic $L$-values}, Duke Math. J. 169 (2020), no. 18, 3439--3476.
		
		
\bibitem{Zud2001} W. Zudilin, \textit{One of the numbers $\zeta(5)$, $\zeta(7)$, $\zeta(9)$, $\zeta(11)$ is irrational}, Uspekhi Mat. Nauk [Russian Math. Surveys] 56 (2001), 149--150 [774--776].
		
\bibitem{Zud2002} W. Zudilin, \textit{Irrationality of values of the Riemann zeta function}, Izvestiya Ross. Akad. Nauk Ser. Mat. [Izv. Math.] 66 (2002), 49--102 [489--542].

\bibitem{Zud2004} W. Zudilin, \textit{Arithmetic of linear forms involving odd zeta values}, J. Th{\'e}or. Nombres Bordeaux 16(1), 251--291 (2004).
		
\end{thebibliography}
\end{document}